\documentclass[12pt]{amsart}
\usepackage[active]{srcltx}
\usepackage[all
]{xy} \xyoption{poly}

\usepackage{xspace}
\usepackage{enumitem}
\usepackage{mathtools}
\usepackage[mathscr]{eucal}
\usepackage{amssymb}
\usepackage{a4wide}
\usepackage{ifthen}
\usepackage{amscd}
\usepackage[hyperindex,bookmarksnumbered,
plainpages,backref]{hyperref}

\usepackage{xcolor}

\newtheorem{theorem}{Theorem}

\newtheorem{lemma}{Lemma}
\newtheorem{remark}{Remark}
\newtheorem{corollary}{Corollary}
\newtheorem{example}{Example}

\makeindex

\begin{document}
	
	{\fontsize{12pt}{17pt}\selectfont
		\title{Normalized Rank- and Determinant-Preserving Mappings of Locally Matrix Algebras}

		\author{Oksana Bezushchak}
		\address{Oksana Bezushchak: Faculty of Mechanics and Mathematics, Taras Shevchenko National University of Kyiv, Volodymyrska, 60, Kyiv 01033, Ukraine}
		\email{bezushchak@knu.ua}
		
		\keywords{Cliﬀord algebra, determinant-preserving mapping, (infinite) tensor product, Jordan homomorphism, locally matrix algebra, preserver, rank-preserving mapping}

		\subjclass[2020]{15A15, 15A86, 47B49 (Primary) 15A30, 15A66, 16W10 (Secondary)}

		\begin{abstract} Let $A$ be a unital locally matrix algebra. Among the examples of such algebras are: (1) an infinite tensor product  $\otimes M_{n_i}(\mathbb{F})$ of matrix algebras over a field $\mathbb{F}$, and (2)~the Clifford algebra of a nondegenerate quadratic form on an infinite-dimensional vector space over an algebraically closed field of characteristic different from~$2$.

We describe linear mappings $A \to B$ between  unital locally matrix algebras that preserve the normalized rank. When $\mathbb{F}$ is a field of real or complex numbers,  we also describe linear mappings  $A \to A$ that preserve the normalized determinant. 	\end{abstract}

				\maketitle
		
		\section{Introduction}

	Let $\mathbb{F}$ be a field and let $M_n(\mathbb{F})$ be the algebra of $n \times n$ matrices over $\mathbb{F}$.  
	For a matrix $a \in M_n(\mathbb{F})$, let $r(a)$ and $\det(a)$ denote the rank and the determinant of $a$, respectively.  
	Let $a^t$ be the transpose matrix of $a$.

	A linear mapping $\varphi \colon M_n(\mathbb{F}) \to M_n(\mathbb{F})$ \textbf{preserves determinant} (respectively \textbf{rank}) if  
	\[
\det(a) = \det(\varphi(a)) \quad \text{( respectively } 	r(a) = r(\varphi(a))  \text{ )}
	\]
	for an arbitrary matrix $a \in M_n(\mathbb{F})$.

	G.~Frobenius~\cite{Frobenius} proved that every determinant-preserving linear mapping is of the type  
	\[
	\varphi(a) = X a Y, \quad a \in M_n(\mathbb{F}), \quad \text{or} \quad 	\varphi(a) = X a^t Y, \quad a \in M_n(\mathbb{F}),
	\]
	where $X, Y$ are invertible $n \times n$ matrices and $	\det(X Y) = 1.$
	
 The result of 	G.~Frobenius was followed by deep generalizations in various directions; see, for example, \cite{Dieudonne,Dolinar_Semrl,Huang_et}.
 
 L.K.~Hua \cite{Hua}, H.G.~Jacob~\cite{Jacob}, and M.~Marcus and B.N.~Moyls~\cite{Marcus_Moyls} proved that every rank-preserving linear mapping $	\varphi:M_n(\mathbb{F}) \to M_n(\mathbb{F})$ is of the type 	\[
 \varphi(a) = X a Y, \quad a \in M_n(\mathbb{F}), \quad \text{or} \quad 	\varphi(a) = X a^t Y, \quad a \in M_n(\mathbb{F}),
 \]
 where $X, Y$ are invertible $n \times n$ matrices.
 
 M.~Bre\v{s}ar and P.~\v{S}emrl~\cite{Bresar_0}, M.~Omladi\v{c} and P.~\v{S}emrl \cite{OmladicSemrl1993RankOne}, and J.~Huang, K.~Kudaybergenov, and F.~Sukochev \cite{Huang_Kudaybergenov_Sukochev}	extended these results to the infinite-dimensional setting of operator algebras.

	Let $A, B$ be associative $\mathbb{F}$-algebras.  
	A linear mapping $\varphi \colon A \to B$ is called a \textbf{Jordan homomorphism} if	$\varphi(a^2) = \varphi(a)^2$	for an arbitrary element $a \in A$. For more information and recent results on Jordan homomorphisms, see~\cite{Bresar_Zelmanov}.

An essential part of the theorems above says that a linear determinant- (respectively, rank-) preserving mapping $\varphi \colon M_n(\mathbb{F}) \to M_n(\mathbb{F})$
	satisfying 	 $\varphi(1) = 1,$ is a Jordan homomorphism.

		In this paper, we extend the descriptions of rank- (respectively, determinant-) preserving mappings to  locally matrix  algebras.

		Recall that an associative $\mathbb{F}$-algebra $A$ is called a \textbf{locally matrix algebra} if every finite subset of $A$ is contained in a subalgebra that is isomorphic to a matrix algebra $M_n(\mathbb{F})$ for some~$n$.  
		An algebra is \textbf{unital} if it contains an identity element.

Let $ \mathbb{N} $ be the set of positive integers.	 A   \textbf{Steinitz number} \cite{ST}   is an infinite formal
	product of the form
\begin{equation}\label{Eq__1} \prod_{p\in \mathbb{P}} p^{r_p}, \end{equation}
	where $ \mathbb{P} $ is the set of all primes, $ r_p \in  \mathbb{N} \cup \{0,\infty\} $ for all $p\in \mathbb{P}$.
	We can define the product of two Steinitz numbers by the rule:
	$$ \prod_{p\in \mathbb{P}} p^{r_p} \cdot  \prod_{p\in \mathbb{P}} p^{k_p}= \prod_{p\in \mathbb{P}} p^{r_p+k_p}, \quad r_p, k_p \in  \mathbb{N} \cup \{0,\infty\},  $$
	where we assume, that
	$$r_p+k_p=\begin{cases}
		r_p+k_p, & \text{if  $r_p < \infty$ and $k_p < \infty$, } \\
		\infty, & \text{in other cases}
	\end{cases}.$$

		For a unital locally matrix algebra $A$ with an identity element $1_A$, consider the set
		\[
		D(A) = \{\, n \ge 1 \mid \text{there exists a subalgebra } A' \text{ such that } 1_A \in A' \subset A, \; A' \cong M_n(\mathbb{F}) \,\}.
		\]  If the $\mathbb{F}$-algebra $A$ is infinite-dimensional, then the set $D(A)$ is infinite. Therefore, the least common multiple of all integers in $D(A)$ is an infinite product of prime numbers, that is, a Steinitz number. We call this least common multiple the \textbf{Steinitz number} of the algebra $A$ and denote it by $\mathbf{st}(A)$  (see~\cite{BezOl,Glimm}).

			J.~Glimm~\cite{Glimm} proved that two unital countable-dimensional locally matrix algebras $A$ and $B$ are isomorphic if and only if $\mathbf{st}(A) =\mathbf{st}(B)$.  
			For algebras of uncountable dimension, this is no longer true; see~\cite{BezOl}.
	
	Consider some examples of unital locally matrix algebras.
			
\begin{example}\label{Ex_1}
	Let $A_i = M_{n_i}(\mathbb{F})$, $i \in I$, be a family of matrix algebras.
	Then the (infinite) tensor product
	\[
	A = \bigotimes_{i \in I} A_i
	\]
	is a unital locally matrix algebra, and
	\[
\mathbf{st}(A) = \prod_{i \in I} n_i.
	\]  \end{example}
	
\begin{example}\label{Ex_2}
	Let $V$ be an infinite-dimensional vector space over an algebraically closed field $\mathbb{F}$, 
	$\operatorname{char} \mathbb{F} \ne 2$.
	Let $f \colon V \to \mathbb{F}$ be a nondegenerate quadratic form.
	Then the Clifford algebra $	\operatorname{Cl}(V, f)$	is a unital locally matrix algebra, and $	\mathbf{st}(\operatorname{Cl}(V, f)) = 2^{\infty}.$ \end{example}
	
 G. K\"{o}the \cite{Koethe}  showed that every countable-dimensional unital locally matrix algebra is isomorphic to an infinite tensor product as in Example~\ref{Ex_1}. These algebras are dense in uniformly hyperfinite $ \mathbb{C}^{*}$-algebras~\cite{Glimm} and have significant applications in physics; see~\cite{Bratteli-Robinson}.

Moreover, a countable-dimensional unital locally matrix algebra of Steinitz number $2^{\infty}$ over an algebraically closed field of characteristic $\neq 2$ can be realized as the Clifford algebra of Example~\ref{Ex_2}.

			Let $A$ be a unital locally matrix algebra.  
			For an element $a \in A$, choose a subalgebra $A' \subset A$ containing  $1_A$ and $ a$ such that  $A' \cong M_n(\mathbb{F})$.	V.~Kurochkin \cite{Kurochkin} noticed that the normalized rank
			\[
			\overline{r}(a) = \frac{r(a)}{n}
			\]
			does not depend on the choice of a subalgebra $A'$. 
			
	Suppose now that $\mathbb{F}$ is  either the field of real numbers or the field of complex numbers. As above, let $A$ be a unital locally matrix algebra.  	For an element $a \in A$, choose a subalgebra $A' \subset A$ containing  $1_A$ and $ a$ such that  $A' \cong M_n(\mathbb{F})$.
Then the \textbf{normalized determinant} is defined by
\[
\overline{\det}(a) = |\det(a)|^{\frac{1}{n}},
\]
and does not depend on the choice of the subalgebra $A'$; see \cite{Bezushchak}.

Thus, unital locally matrix algebras are equipped with well-defined normalized rank
and, assuming $\mathbb{F}=\mathbb{R}$ or $\mathbb{F}=\mathbb{C}$, the normalized determinant functions.

Nonunital locally matrix algebras have been studied in~\cite{Baranov2,Spectra_Bezushchak,Diskme}.
The simplest example of such algebras is the algebra of infinite finitary matrices.

			Suppose that $\dim_{\mathbb{F}} A = \aleph_0$.  
			Let $1 \in M_{n_1}(\mathbb{F}) \subset M_{n_2}(\mathbb{F}) \subset \cdots$  
			be an ascending chain of matrix algebras such that
			\[
			\bigcup_{i \ge 1} M_{n_i}(\mathbb{F}) = A.
			\]
			The transpose mapping of $M_{n_i}(\mathbb{F})$ extends to the transpose mapping on $M_{n_{i+1}}(\mathbb{F})$.  
			Hence the algebra $A$ is equipped with a transpose mapping	$t : A \to A.$

			A linear mapping of associative algebras $\psi : A \to B$ is called an \textbf{antihomomorphism} if
$	\psi(a_1 a_2) = \psi(a_2) \psi(a_1)$	for arbitrary elements $a_1, a_2 \in A$.  
			The transpose is an example of an antihomomorphism.
			
		 \begin{theorem}\label{Th__1} 
			Let $A$ and $ B$ be unital locally matrix algebras over a field $\mathbb{F}$ whose characteristic is not $2$.  
			Let $\varphi : A \to B$ be a linear mapping that preserves the normalized rank.  
			Then there exist orthogonal idempotents $e_1, e_2 \in B$, $e_1 + e_2 = 1_B;$  
			a homomorphism $\varphi_1 : A \to e_1 B e_1$, $\varphi_1(1_A) = e_1$;  
			an antihomomorphism $\varphi_2 : A \to e_2 B e_2$, $\varphi_2(1_A) = e_2$;  
			and invertible elements $X, Y \in B$ such that
			\[	\varphi(a) = X \bigl( \varphi_1(a) + \varphi_2(a) \bigr) Y.
			\]  \end{theorem}
			
 \begin{remark}  
			For countable-dimensional unital locally matrix algebras, their unital endomorphisms were described in~\cite{14}. \end{remark}
			
 \begin{corollary} \label{Cor__1} 
			If a surjective linear mapping $\varphi : A \to B$ preserves the normalized rank,  
			then there exist an isomorphism or an antiisomorphism $	\psi : A \to B$			and invertible elements $X, Y \in B$ such that 	$\varphi(a) = X \, \psi(a) \, Y, $ $ a \in A.$ \end{corollary}
		
		It turns out that the description of normalized rank-preserving linear mappings 
		$A \to A$ depends on the Steinitz number   $\mathbf{st}(A)$.  
		
A Steinitz number   (\ref{Eq__1})  is called \textbf{locally finite} if $k_p < \infty$ for all prime numbers $p$.
		
	 \begin{corollary} \label{Cor__2} 
		If the Steinitz number $\mathbf{st}(A)$ is locally finite, 
		then every normalized rank-preserving linear mapping $ A \to A$ 
		is either a homomorphism or an antihomomorphism.  
		If $\mathbf{st}(A)$ is not locally finite and $\dim_{\mathbb{F}} A = \aleph_0$, 
		then there exists a normalized rank-preserving linear mapping $ A \to A$ 
		that is neither a homomorphism nor an antihomomorphism. \end{corollary}

	\begin{theorem}\label{Th__2}
		Let $A, B$ be unital locally matrix algebras over a field 
		$\mathbb{F} = \mathbb{R}$ or $\mathbb{C}$.  
		Let $\varphi \colon A \to B$ be a surjective linear mapping that preserves 
		the normalized determinant.  
		Then $\varphi$ is of one of the following types:
		\[
		\varphi(a) = X \, \varphi_1(a) \, Y, \quad a \in A,
		\]
		where $\varphi_1 \colon A \to B$ is an isomorphism, or
		\[
		\varphi(a) = X \, \varphi_2(a) \, Y, \quad a \in A,
		\]
		where $\varphi_2 \colon A \to B$ is an anti\-isomorphism;  
		here $X, Y \in B$ are invertible elements with $\overline{\det}(X Y) = 1$. 	\end{theorem}
		
		\section{Rank-preserving mappings}

		An element $e$ of an associative algebra is called an \textbf{idempotent} 
		if $e^2 = e$.

	\begin{lemma}\label{Lem__1} 
	Let $\varphi \colon M_n(\mathbb{F}) \to M_m(\mathbb{F})$   be a linear
	mapping that preserves the normalized rank and maps the identity matrix $I_n$
	to the identity matrix $I_m$. Then $\varphi$ is a Jordan homomorphism.  \end{lemma}

	\begin{proof} It is well known that a linear transformation $A$ on an $n$-dimensional space  $V$ is idempotent if and only if $V = \operatorname{Ker}(A) \oplus  \operatorname{Im}(A).$ In matrix terms, this means that a matrix $a\in M_n(\mathbb{F}) $ is idempotent if and only if
	\[
	r(a)+r(I_n-a)=n,
	\]
	or, equivalently,
	\[
	\overline{r}(a)+\overline{r}(I_n-a)=1.
	\]
	This implies that the image $\varphi(a)$ of an idempotent $a\in M_n(\mathbb{F})$
	is an idempotent in $M_m(\mathbb{F})$.

	M.~Bre\v{s}ar and P.~\v{S}emrl~\cite{Bresar_00} proved that if a linear mapping from
	$M_n(\mathbb{F})$ to an arbitrary associative $\mathbb{F}$-algebra maps
	idempotents to idempotents, then it is a Jordan homomorphism.
	This completes the proof of the lemma.  \end{proof}

\begin{proof}[Proof of Theorem~\ref{Th__1}] 		Let $A' \subset A$ be a subalgebra of $A$ such that $1_A \in A'$ and $A' \cong M_n(\mathbb{F})$.
		Choose a subalgebra $B' \subset B$ such that $\varphi(A') \subseteq B'$ and $B' \cong M_m(\mathbb{F})$.
		By Lemma~\ref{Lem__1}, the restriction $\varphi|_{A'} \colon A' \to B'$ 
		is a Jordan homomorphism. This implies that 
		$\varphi \colon A \to B$ is a Jordan homomorphism.

For an arbitrary integer $n \in D(A)$, choose a subalgebra $A' \subset A$ containing  $1_A$ such that  $A' \cong M_n(\mathbb{F})$.	Let $C_A(A')$ be the centralizer of $A'$ in the algebra $A$. It is well known (see~\cite{Drozd_Kirichenko}) that $A$ is isomorphic
to the tensor product $A'\otimes_{\mathbb{F}} C_A(A')$; hence
\[
A\cong M_n\!\bigl(C_A(A')\bigr).
\]	Now Theorem~\ref{Th__1} follows immediately  from the result of N.~Jacobson and C.~Rickart~\cite{Jacobson2}
		on Jordan homomorphisms of matrix algebras.
		This completes the proof of Theorem~\ref{Th__1}. \end{proof}

		\begin{proof}[Proof of Corollary~\ref{Cor__1}] 	Suppose that a surjective linear mapping $\varphi \colon A \to B$ preserves the normalized rank.
		If the idempotents $e_1, e_2 \in A$ (see Theorem~\ref{Th__1}) are both nonzero, then
		\[
		\varphi(A) \subseteq e_1 B e_1 + e_2 B e_2, \quad \varphi(A) = B.
		\]
		But $B \not= e_1 B e_1 \oplus e_2 B e_2$, since the algebra $B$ is simple.
		This contradiction proves the corollary. \end{proof}
		
		\begin{proof}[Proof of Corollary~\ref{Cor__2}]  	Suppose that the Steinitz number $s = \mathbf{st}(A)$ is locally finite, and that $e_1, e_2 \in A$ are nonzero orthogonal idempotents satisfying $e_1 + e_2 = 1_A$.
			Let $\varphi_1 \colon A \to e_1 A e_1$ be a homomorphism with $\varphi_1(1_A) = e_1$, and let $\varphi_2 \colon A \to e_2 A e_2$ be an antihomomorphism with $\varphi_2(1_A) = e_2$.

			If $1_{A} \in A' \subset A$ is a locally matrix subalgebra of $A$, then $\mathbf{st}(A')$ divides $\mathbf{st}(A)$. 	Hence,	$s \mid \mathbf{st}(e_1 A e_1)$ and $ s \mid \mathbf{st}(e_2 A e_2).$		Let 
		\[
		s = \prod_{p \in \mathbb{P}} p^{k_p}, \quad k_p < \infty.
		\]
		There exists a subalgebra $M_n(\mathbb{F}) \subset A$ such that 
		$e_1, e_2 \in M_n(\mathbb{F})$, and
		\[
		n = p_1^{k_1} p_2^{k_2} \cdots p_t^{k_t}.
		\]
		Let $r_i$ be the rank of the idempotent $e_i$ in the matrix algebra $M_n(\mathbb{F})$, $i = 1, 2$.  
		In~\cite{Bezushchak_Oliynyk_1} we showed that 
		\[
		\mathbf{st}(e_i A e_i) = \frac{r_i}{n}\, s.
		\]
		Since $\varphi_1 \colon A \to e_1 A e_1$ is a unital embedding, 
		it follows that $s \mid \tfrac{r_1}{n} s$.  
		This is possible only if $r_1 = n$, a contradiction.  
		We have proved that either $e_1 = 1_A$ or $e_2 = 1_A$.  
		In the first case, $\varphi$ is a homomorphism;  
		in the second case, $\varphi$ is an antihomomorphism. 
		
		\begin{remark}
	It is straightforward to see that  every unital (anti)homomorphism 
		$\varphi \colon A \to A$ preserves the normalized rank.  
		A.~Kurosh~\cite{Kurosh} proved that every countable-dimensional unital locally matrix algebra 
		has a unital endomorphism that is not surjective.  
		For another proof and a description of such endomorphisms, see~\cite{14}. 	\end{remark}

		Suppose now that the algebra $A$ is countable-dimensional and $\mathbf{st}(A) = p^{\infty} \cdot s'$.  
		Choose a subalgebra $1 \in M_p(\mathbb{F}) \subset A$.  
		Then $A \cong M_p(A')$, where $A'$ is the centralizer of the subalgebra $M_p(\mathbb{F})$.
		We have $\mathbf{st}(A) = p \cdot \mathbf{st}(A')$, hence $\mathbf{st}(A)  = \mathbf{st}(A') $. By J.~Glimm’s theorem, $A \cong A'$.  	Let $\psi \colon A \to A'$ be an isomorphism.

		We have shown above that every countable-dimensional unital locally matrix algebra $A$ 	has an antiautomorphism $a \mapsto a^t$, $a \in A$.  	The unital homomorphism
		\[
		A \to A, \quad a \to \operatorname{diag}\bigl(	\underbrace{\psi(a), \ldots, \psi(a)}_{p-1}, \psi(a^t)\bigr)
		\]
		is neither a homomorphism nor an antihomomorphism.  
		This completes the proof of Corollary~\ref{Cor__2}. \end{proof}

		\section{Determinant-preserving mappings: proof of Theorem \ref{Th__2}}

		In this section, we assume that $\mathbb{F}$ is the field of real  	or complex numbers.  
		Let $\deg f(t)$ denote the degree of a polynomial $f(t) \in \mathbb{F}[t]$.

		\begin{lemma}\label{Lem__2}
		Let $a \in M_n(\mathbb{F})$ be an $n\times n$  matrix.
		Then
		\[
		r = r(a) = \max \{ \deg \det (ta + b) \mid b \in M_n(\mathbb{F}) \}.
		\] \end{lemma}

			\begin{proof}	Both the left-hand side and the right-hand side do not change under elementary operations on rows and columns. Hence, without loss of generality, we can assume that
		\[
		a = \operatorname{diag}(\lambda_1, \ldots, \lambda_r, 0, \ldots, 0),
		\quad \lambda_i \neq 0, \; 1 \le i \le r.
		\]
		For an arbitrary matrix $b = (b_{ij})_{n \times n} \in M_n(\mathbb{F})$, we have
		\[
		\det(ta + b) = t^r \lambda_1 \cdots \lambda_r \cdot 
		\det
		\begin{bmatrix}
			b_{r+1, r+1} & \cdots & b_{r+1, n} \\
			\vdots & \ddots & \vdots \\
			b_{n, r+1} & \cdots & b_{n,n}
		\end{bmatrix}
		+ \sum_{i < r} t^i (\cdots).
		\]
		This implies the assertion of the lemma. 	\end{proof}

		\begin{lemma}\label{Lem__3}		We have 
				\begin{equation}\label{Eq__2} \overline{r}(a) \le \overline{r}(\varphi(a)) 	\end{equation}
					 for an arbitrary element $a \in A$.	\end{lemma}

		\begin{proof} 		For an element $a \in A$, choose a subalgebra $A' < A$ containing  both the identity element $1_A$ and the element $ a$, and such that  $A' \cong M_n(\mathbb{F})$.	There exists a subalgebra $B' \subset B$ such that $\varphi(A') \subseteq B'$ and $B' \cong M_m(\mathbb{F})$.
		We need to show that
		\[
		\frac{r(a)}{n} \le \frac{r(\varphi(a))}{m}.
		\]
		
		First, we note that
		\begin{equation}\label{Eq__3}
		\det(t a + a')^m = \det \bigl( t \varphi(a) + \varphi(a') \bigr)^n
			\end{equation}
		in $\mathbb{F}[t]$ for arbitrary elements $a, a' \in A'$.
		Indeed, both sides of the equality (\ref{Eq__3}) are polynomials and the field $\mathbb{F}$ is infinite. For an arbitrary $\alpha \in \mathbb{F}$, we have
		\[
		\det(\alpha a + a')^m = \det \bigl( \alpha \varphi(a) + \varphi(a') \bigr)^n.
		\]
		By Lemma~\ref{Lem__2}, 	$r(a) = \max \{ \deg \det (t a + a') \mid a' \in A' \}.$	Hence,
		\[
		m \, r(a) = \max \{ \deg \det (t a + a')^m \mid a' \in A' \}.
		\]
		By the above,
		\[
		\det (t a + a')^m = \det \bigl( t \varphi(a) + \varphi(a') \bigr)^n.
		\] 		Hence,
		\[
		m\cdot r(a) 
		= n \cdot \max \{ \deg \det (t \varphi(a) + \varphi(a')) 
		\mid a' \in A \}
		\le \] \[  n \cdot \max \{ \deg \det (t \varphi(a) + b) 
		\mid b \in B' \} 
		= n \cdot  r(\varphi(a)).
		\]
		This completes the proof of the lemma. \end{proof}

		\begin{corollary}\label{Cor__3} 	The mapping $\varphi$ is a bijection. 	\end{corollary} 
		
		Indeed, the inequality (\ref{Eq__2})	implies that $\ker \varphi = \{0\}$.

		The inverse linear mapping $\varphi^{-1}$ also preserves the normalized determinant. Hence
		\[
		\overline{r}(\varphi(a)) \le \overline{r}(a), \quad a \in A,
		\]
		and therefore
		\[
		\overline{r}(a) = \overline{r}(\varphi(a)).
		\]	By Theorem~\ref{Th__1}, there exists an isomorphism or an antiisomorphism
		$\psi \colon A \to B$ and invertible elements $X, Y \in B$
		such that $\varphi(a) = X \psi(a) Y$ for all $a \in A$.
		Choosing $a = 1_A$, we get $X Y  = \varphi(1_A)$,
		which implies $\overline{\det}(X Y) = 1$, 
		and this completes the proof of Theorem~\ref{Th__2}.

		\section*{Acknowledgment}

				\medskip
		
The author thanks Karimbergen Kudaybergenov for drawing attention to this problem and is grateful to Matej Bre\v{s}ar for valuable discussions.

	\end{document}